\documentclass[11pt]{amsart}

\usepackage[a4paper,hmargin=3cm,vmargin=3cm]{geometry}
\usepackage{amsfonts,amssymb,amscd,amstext}
\usepackage{graphicx}
\usepackage[dvips]{epsfig}
\usepackage[ansinew]{inputenc}

\usepackage{fancyhdr}
\usepackage{times}

\usepackage{enumerate}
\usepackage{titlesec}
\usepackage{mathrsfs}
\usepackage{stmaryrd}

\def\R{\mathbb{R}}

\def\Z{\mathbb{Z}}

\def\D{\mathbb{D}}

\newcommand{\ben}{\begin{enumerate}}
\newcommand{\bit}{\begin{itemize}}
\newcommand{\een}{\end{enumerate}}
\newcommand{\eit}{\end{itemize}}

\newcommand{\ed}{\end{document}}

\def\cU{\mathcal{U}}

\def\cH{\mathcal{H}}

\def\cM{\mathcal{M}}

\let\8=\infty \let\0=\emptyset  
\let\hat=\widehat

\let\landa=\lambda
\let\alfa=\alpha

\let\parc=\partial

\def\ep{\varepsilon}

\def\landa{\lambda}

\def\flecha{\rightarrow}

\def\cte.{\mathop{\rm cte.}\nolimits}

\def\R{\mathbb{R}}
\def\Z{\mathbb{Z}}

\def\D{\mathbb{D}}

\def\S{\mathbb{S}}

\titleformat{\section}%[display]
{\filcenter\bfseries\large} {\thesection{.}}{0.2cm}{}%[$\vspace*{-1.0cm}$]
\titleformat{\subsection}[runin]
{\bfseries} {\thesubsection{.}}{0.15cm}{}[.]
\titleformat{\subsubsection}[runin]
{\em}{\thesubsubsection{.}}{0.15cm}{}[.]
\usepackage[up,bf]{caption}

\newtheorem{theorem}{Theorem}[section]

\newtheorem{proposition}[theorem]{Proposition}

\newtheorem{corollary}[theorem]{Corollary}

\newtheorem{example}[theorem]{Example}

\theoremstyle{definition}
\numberwithin{equation}{section}
\numberwithin{figure}{section}

\usepackage{color}

\begin{document}
\fancyhead[LO]{Serrin's overdetermined problem}
\fancyhead[RE]{José A. Gálvez, Pablo Mira}
\fancyhead[RO,LE]{\thepage}

\thispagestyle{empty}

%% Title
%\vspace*{7mm}
\begin{center}
{\bf \LARGE Serrin's overdetermined problem for fully nonlinear\\[0.2cm] non-elliptic equations}
\vspace*{5mm}

%% Authors
\hspace{0.2cm} {\Large José A. Gálvez, Pablo Mira}
\end{center}

%% Addresses and finantial support
\footnote[0]{\vspace*{-0.4cm} \emph{Mathematics Subject Classification}: 35J25, 49Q05, 53A10}
%}
%% Abstract, keywords, and MSC

\vspace*{7mm}

\begin{quote}
{\small
\noindent {\bf Abstract}\hspace*{0.1cm}
Let $u$ denote a solution to a rotationally invariant Hessian equation $F(D^2u)=0$ on a bounded simply connected domain $\Omega\subset \R^2$, with constant Dirichlet and Neumann data on $\parc \Omega$. In this paper we prove that if $u$ is real analytic and not identically zero, then $u$ is radial and $\Omega$ is a disk. The fully nonlinear operator $F\not\equiv 0$ is of general type, and in particular, not assumed to be elliptic. We also show that the result is sharp, in the sense that it is not true if $\Omega$ is not simply connected, or if $u$ is $C^{\8}$ but not real analytic.% we drop the analyticity assumption.
%We prove that any real analytic solution $u$ of a rotationally symmetric Hessian equation $F(D^2u)=0$ on a simply connected bounded domain $\Omega\subset \R^2$ that satisfies the overdetermined Serrin-type conditions $u=0$ and $|Du|={\rm const}$ on $\parc \Omega$, is a radial function.

%\vspace*{0.1cm}
%\noindent{\bf Keywords}\hspace*{0.1cm} Complex curves.

\vspace*{0.1cm}

%\noindent{ MSC (2010): 49Q05, 53A10.}\hspace*{0.1cm}
}
\end{quote}

%%%%%%%%%%
%%%%%%%%%%
%%%%%%%%%%
%%%%%%%%%% INTRODUCTION
%%%%%%%%%%
%%%%%%%%%%

\section{Introduction}
Let $\Omega\subset \R^{n}$ be a $C^2$ smooth bounded domain, and let $u\in C^2(\overline{\Omega})$ be a solution to $\Delta u+1=0$ that satisfies overdetermined boundary conditions
 \begin{equation}\label{over}
 u=0 , \hspace{1cm} |Du|={\rm constant}  \hspace{0.5cm} \text{ on $\parc \Omega$}.
 \end{equation}
In his famous 1971 paper \cite{Se}, Serrin proved that in these conditions $\Omega$ is a ball and $u$ is a radial function. Starting with Serrin's paper (see also the influential work by Weinberger \cite{W}), there has been a great interest in extending Serrin's result to more general PDEs that satisfy the overdetermined conditions \eqref{over}. %Providing a complete list of references in this respect is virtually impossible, but some relevant ones are [...]. 
Ellipticity has typically been an essential component in all these extensions of Serrin's theorem.%, and without it the desired radial symmetry is not true in general. %Some simple counterexamples will be provided in Section \ref{sec:defi}.

%Essentially all these extensions of Serrin's result deal with \emph{elliptic} equations (as we will see in Section \ref{sec:defi}, unsurprisingly, Serrin's result is not true in general without some ellipticity hypothesis)

In this paper we consider Serrin's overdetermined problem for general (not necessarily elliptic) fully nonlinear Hessian equations, i.e. 
\begin{equation}\label{serrineq}\def\arraystretch{1.6}\left\{\begin{array}{lll} F(D^2u)=0 & \text{ in } & \Omega, \\
u=0, \hspace{0.5cm} |Du|=c & \text{ on } & \parc \Omega, %\\
%& \text{ on } & \parc\Omega.
\end{array} \right.
\end{equation} 
where $F$ is a function on the space $\cM_n$ of all symmetric $n\times n$ matrices. To avoid meaningless situations, \emph{we will assume from now on that $F$ is never locally zero}, i.e. $F\not \equiv 0$ on any open set. A natural and necessary hypothesis on $F$ dictated by the nature of the boundary conditions \eqref{over} (see e.g. Silvestre and Sirakov \cite{SS}) is that $F$ is \emph{rotationally invariant}, i.e. $F(Q^t M Q)=F(M)$ for any $M\in \cM_n$ and any orthogonal matrix $Q$. Equivalently, $F$ is a symmetric function of the eigenvalues of the Hessian $D^2u$. 

When $F$ is elliptic in a suitable sense,  the existence of a solution $u$ to \eqref{serrineq} forces $\Omega$ to be a ball, and $u$ to be a radial function (see \cite{SS}). It is not surprising that, if $F$ is not elliptic, this is not true anymore; some simple counterexamples will be given in Example \ref{eje1}. So, in some sense, the rigidity given by ellipticity seems fundamental for the desired radial symmetry result to hold.

This situation makes our main result here somehow unexpected. We prove that if $\Omega\subset \R^2$ is simply connected, and $u$ is a real analytic solution to \eqref{serrineq}, then $\Omega$ is a disk and $u$ is radial. No ellipticity assumption is made on $F$, and no sign assumption is made on $u$:

%Our goal in this paper is to extend Serrin's paper to general (not necessarily ellipitic) second-order fully nonlinear Hessian equations, i.e. equations of the form $F(D^2 u)=0$. For this, due to the nature of the boundary conditions \eqref{over}, a natural and necessary assumption (see e.g. Silvestre and Sirakov \cite{SS}) is that $F$ is \emph{rotationally invariant}, i.e. $F$ is a function on the space $\cM_n$ of all symmetric $n\times n$ matrices, and $F(Q^t M Q)=F(M)$ for all such $M$ and any orthogonal matrix $Q$. Equivalently, $F$ is a symmetric function of the eigenvalues of the Hessian $D^2u$. %$\Delta u$ and of the Hessian determinant ${\rm det}(D^2u)$.

%It follows by \cite{SS} that when $F$ is rotationally invariant, smooth and {\bf elliptic} in an adequate sense, then the existence of a solution $u$ to 
%\begin{equation}\label{serrineq}\def\arraystretch{1.6}\left\{\begin{array}{lll} F(D^2u)=0 & \text{ in } & \Omega, \\
%u=0, \hspace{0.5cm} |Du|=c & \text{ on } & \parc \Omega, %\\
%%& \text{ on } & \parc\Omega.
%\end{array} \right.
%\end{equation} 
%forces the domain $\Omega$ to be a ball and $u$ to be a radial function. 
\begin{theorem}\label{main}
Let $\Omega\subset \R^2$ be a smooth bounded, simply connected domain, let $F:\cM_2\flecha \R$ be rotationally invariant, and let $u\in C^{\omega}(\overline{\Omega})$ be a non-zero solution to \eqref{serrineq}. 

Then $\Omega$ is a disk and $u$ is a radial function with respect to the center of $\Omega$.
\end{theorem}
Remarkably, the topological hypothesis that $\Omega$ is simply connected cannot be weakened. Indeed, in Section \ref{sec:defi} we will show that there exist positive real analytic solutions $u$ to \eqref{serrineq} for which $u$ is non-radial and $\Omega\subset \R^2$ is diffeomorphic to an annulus. This example also shows the global nature of Theorem \ref{main}, and in particular indicates that it cannot follow from a local application of the Cauchy-Kowalevsky theorem along the boundary. In addition, there exist non-radial, $C^{\8}$ solutions $u$ to \eqref{serrineq} for $\Omega$ simply connected (see Example \ref{eje1}). Thus, Theorem \ref{main} is sharp in these directions.

Theorem \ref{main} is inspired by classical surface theory, and in particular by a theorem of K. Voss \cite{Vo} according to which any compact analytic Weingarten surface of genus zero immersed in $\R^3$ is a rotational sphere. The proof of Theorem \ref{main} is an application of the Poincaré-Hopf theorem to an adequate line field with singularities in $\overline{\Omega}$. We emphasize that this line field is not given by the gradient of a solution to $F(D^2 u)=0$, so in this sense the application of the Poincaré-Hopf theorem here is not very usual in PDE theory. This strategy was used by the second author in \cite{Mi} in order to solve overdetermined problems with non-constant boundary data for fully nonlinear elliptic equations, and also by Espinar and Mazet in \cite{EM} for solving the classification problem of $f$-extremal disks in the two-sphere $\S^2$. Both of these works are inspired by our previous paper \cite{GM3}, about uniqueness of immersed spheres modeled by elliptic PDEs in three-manifolds. The key tool in all these works is to use ellipticity in order to construct a line field on the surface with isolated singularities of negative index, and derive from there a contradiction with the Poincaré-Hopf theorem. However, in our present situation, the lack of ellipticity makes this approach unsuitable; the natural line field that we construct may have non-isolated singularities, and even at the isolated ones its index can be positive. 

The results in the present paper strengthen the connection between overdetermined problems and hypersurface theory, a connection already present in Serrin's theorem, and that has been exploited in many works, see e.g. \cite{DPW,DEP,EFM,EM,FV,FMV,HHP,Mi,RRS1,RRS2,RS,ScS,Si,T}. Nonetheless, to the authors' best knowledge, Theorem \ref{main} might be the first example of such connection for non-elliptic equations.

We next provide an outline of the proof of Theorem \ref{main}. Let $u$ be a non-radial, real analytic solution to \eqref{serrineq} on $\overline\Omega$. Then, the eigenlines of $D^2u$ define two analytic line fields $L_1,L_2$ on $\overline{\Omega}-\cU$, where $\cU$ is the set of points in $\overline\Omega$ where $D^2u$ is proportional to the identity, i.e. the set of points where $D^2 u$ has a double eigenvalue. We wish to analyze how $L_1,L_2$ extend across $\cU$.

In order to do this, we consider a point $p\in \cU$, and we let $w(x,y)$ be the first non-zero homogeneous polynomial of degree $n\geq 3$ in its series expansion around that point (if $w$ does not exist, the result is trivial). There will be three cases to consider.

If $w(x,y)$ is radially symmetric, we will prove in Proposition \ref{prop1} that $u(x,y)$ is also radially symmetric with respect to $p$, up to a translation. In that case, the result follows easily, and so we discard this situation. If $w(x,y)$ is some power of a linear function, we will prove in Section \ref{sec:ap} that $\cU$ is a real analytic regular curve around $p$, and that the eigenfields $L_1,L_2$ extend analytically across $p$. Finally, if $w(x,y)$ is not of any of these two types, then we will show in Section \ref{sec:defi} that it is a harmonic polynomial; in that case $p$ is isolated in $\cU$, and the Poincaré-Hopf index of the line fields $L_1,L_2$ around $p$ is negative. 

Once there, the proof ends as follows. By the previous discussion, both $L_1,L_2$ can be extended to line fields on $\overline{\Omega}$ with only isolated singularities, all of them of negative index. Also, the overdetermined conditions \eqref{over} imply that one among $L_1$ or $L_2$ is tangent to $\parc \Omega$ at each boundary point. Since $\Omega$ is simply connected, this provides a contradiction with the Poincaré-Hopf index theorem. The contradiction proves that $u$ is radial, and from there, that $\Omega$ is a disk.

% ensures that sum of the indices of a line field on $\overline{\Omega}$ that is tangent to $\parc \Omega$ and only has isolated singularities Associated to any non-canonical solution to \eqref{prob}, a line field on $\overline{\Omega}$ with isolated singularities of negative index that is tangent to $\parc \Omega$, and to derive from there a contradiction with the Poincaré-Hopf theorem using that $\Omega$ is simply connected. 

\section{Necessity of the hypotheses}\label{sec:defi}

We will first show that the hypothesis that $u$ is real analytic in Theorem \ref{main} is necessary, by constructing a $C^{\8}$ solution to \eqref{serrineq}  that is not a radial function. In this construction, the domain $\Omega\subset \R^2$ is an arbitrary simply connected smooth (or even real analytic) bounded domain.

\begin{example}\label{eje1}
Let $\Omega\subset \R^2$ be the simply connected domain bounded by a real analytic regular Jordan curve $\gamma$ in $\R^2$. Given any $\rho>0$, let $f$ denote a smooth function on the closed disk $D_{\rho}=\overline{D(0,\rho)}$, with the following properties:
 \begin{enumerate}
 \item
$f$ is a radial function with respect to the origin. 
\item
The value of $f$ and all its derivatives vanish at every point of $\parc D_{\rho}$.
 \end{enumerate}
For instance, we can choose $$f(x,y)=\exp\left(\frac{-1}{\rho^2-(x^2+y^2)}\right).$$% for $a=\exp(-1/\rho^2).$

Let now $D_1,\dots, D_k$ denote a collection of mutually disjoint closed disks in $\Omega$, of radius $\rho>0$. For each $i\in \{1,\dots, k\}$, let $u_i\in C^{\8}(D_i)$ be given by $$u_i(x,y):= f(x-a_i,y-b_i)$$ where $(a_i,b_i)$ is the center of $D_i$. 

%
%%For each $i\in \{1,\dots, k\}$, let $u_i$ denote a function on $D_i$ with the following properties:
%\begin{enumerate}
%\item
%$u\in C^{\8}(D_i)$. 
%\item
%$u$ is a radial function with respect to the center of $D_i$.
% \item
%The value of $u$ and all its derivatives vanish at every point of $\parc D_i$.
%\end{enumerate}
%%For instance, the function $$\exp\left(\frac{-1}{1-(x^2+y^2)}\right)$$ satisfies all these assumptions in $D:=D(0,1)$, and it is easy to transform it to be defined on any of the disks $D_i$. 
Define now the function $u\in C^{\8}(\overline{\Omega})$ by 
$$u(x,y)=u_i(x,y) \hspace{0.3cm} \text{ if } \hspace{0.3cm} (x,y)\in D_i,\hspace{1cm} u(x,y)=0 \text{ otherwise}.$$ 

\noindent Note that $u=|Du|=0$ along $\parc \Omega$, i.e. $u$ satisfies the overdetermined boundary conditions in \eqref{serrineq} for the choice $c=0$.

Since each function $u_i$ is radial with respect to some point in $\R^2$, it follows that the Jacobian

\begin{equation}\label{jaceroo}
J[\Delta u, \cH(u)]:= (\Delta u)_x (\cH(u))_y-(\Delta u)_y (\cH(u))_x=0
\end{equation}

\noindent on $\Omega$, where we are denoting $\cH(u):={\rm det}(D^2u)$. This implies by a classical theorem of Brown and Sard, see e.g. \cite{new}, that there exists a smooth function $\Phi(s,t)$ with $\Phi(0,0)=0$ such that 

\begin{equation}\label{phieq}
\Phi(\Delta u, \cH(u))=0.
\end{equation}

\noindent By considering the function $F\in C^\8(\mathcal{M}_2)$ associated to $\Phi$ in the obvious way, this implies that $u$ is a non-radial $C^{\8}$ solution to \eqref{serrineq}, for $c=0$ and the above choices of $\Omega$ and $F$.

\end{example}

We remark that, even though the function $u$ constructed in Example \ref{eje1} is not real analytic, the function $\Phi$ (and so, the function $F$ in \eqref{serrineq}) can be chosen to be real analytic in many cases. 

The next example shows that the hypothesis that $\Omega$ is simply connected in Theorem \ref{main} cannot be removed. For that, we will construct real analytic, non-radial solutions to overdetermined problems of the form \eqref{serrineq}, with $F\not\equiv 0$ real analytic and rotationally symmetric, on planar domains $\Omega\subset \R^2$ diffeomorphic to an annulus but that are not radially symmetric in general.

\begin{example}\label{eje2}
Let $\gamma(s):=(\alfa(s),\beta(s))$, $s\in[0,L]$, be a real analytic, regular Jordan curve in $\R^2$ parametrized by arc-length, and let $\nu(s)$ be its unit normal. Assume that the normal map $$\Psi(s,t):= \gamma(s)+t\nu(s) : \R/(L\Z) \times[-1,1] \flecha \R^2$$ is a real analytic diffeomorphism onto the compact planar region $$\overline{\Omega}:=\{\gamma(s)+t\nu(s) : |t|\leq 1\}\subset \R^2.$$ Note that $\overline{\Omega}$ is real analytic and diffeomorphic to an annulus. We remark that the condition that $\Psi$ is a local diffeomorphism is equivalent to the curvature $\kappa(s)$ of $\gamma$ satisfying $|\kappa|<1$ at every point. If the curve $\gamma(s)$ is chosen to be convex, the condition $|\kappa|<1$ is also sufficient for $\Psi$ being a real analytic diffeomorphism.

%\begin{enumerate}
%\item
%The curvature $\kappa(s)$ of $\gamma$ satisfies $|\kappa|<1$ at every point.
% \item
%The normal map $(s,t)\mapsto \gamma(s)+t\nu(s)$ for $|t|\leq 1$ is injective, and thus the normal domain $$\Omega:=\{\gamma(s)+t\nu(s) : |t|\leq 1\}\subset \R^2$$ is diffeomorphic to an annulus.
%\end{enumerate}

Define $u\in C^{\omega}(\overline{\Omega})$ by $$u(\Psi(s,t))=1-t^2.$$ A computation using $x=\alfa(s)-t\beta'(s)$, $y=\beta(s)+t\alfa'(s)$ shows that
\begin{equation}\label{ecuu}
u_{xx}+u_{yy}= -2 +\frac{2t\kappa(s)}{1-t \kappa(s)}, \hspace{1cm} u_{xx} u_{yy}-u_{xy}^2=\frac{-4t \kappa(s)}{1-t \kappa(s)}.\end{equation}

\noindent Let $\phi_1(s,t)$ and $\phi_2(s,t)$ denote the right-hand sides in \eqref{ecuu}. A computation shows that the Jacobian determinant of the map $(s,t)\mapsto (\phi_1,\phi_2)$ vanishes identically. Consequently, by \eqref{ecuu}, we have that $u(x,y)$ satisfies \eqref{jaceroo}. Arguing as in Example \ref{eje1}, we conclude that $u$ is a solution to \eqref{phieq} on $\overline{\Omega}$, for some real analytic function $\Phi$.

Finally, we note that $u$ satisfies the overdetermined boundary conditions \eqref{over}. That $u=0$ along $\parc \Omega$ is clear by construction. A computation shows that $Du=(u_x,u_y)=(2t\beta'(s),-2t \alfa'(s))$, which has constant length for each fixed value of $t$. In particular, $|Du|=2$ along $\parc \Omega$, which corresponds to $t=\pm 1$.

To sum up, $u\in C^{\omega}(\overline{\Omega})$ is a real analytic, non-radial solution to \eqref{serrineq} for the real analytic choice of $\Phi$ above, and for the real analytic annulus $\Omega$ in $\R^2$. Note that, in general, $\Omega$ is not radially symmetric, and the two closed curves in $\parc \Omega$ are not necessarily convex.

\begin{figure}[h]
\begin{center}
\includegraphics[width=10cm]{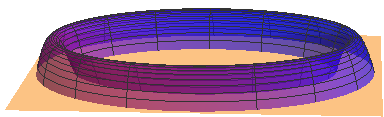}
\caption{A non-radial, analytic solution to \eqref{serrineq} in an elliptical region, with base on the ellipse $4x^2+y^2=64$.} \label{compara}
\end{center}
\end{figure}

\end{example}

\section{Behavior around points with a double eigenvalue of $D^2u$}\label{sec:main}

Let $u=u(x,y)$ denote a real analytic solution to $F(D^2 u)=0$ on a regular planar domain $\overline\Omega\subset \R^2$, where $F:\cM_2\flecha \R$ is rotationally invariant. Denoting $\cH(u):={\rm det}(D^2 u) = u_{xx} u_{yy}-u_{xy}^2$, this equation can be rewritten as 
 \begin{equation}\label{weq}
 \Phi(\Delta u, \cH(u))=0,
 \end{equation}
where $\Phi$ is not identically zero on any open set of $\R^2$, by the related hypothesis on $F$. Denote 
 \begin{equation}\label{nees}
\cU= \{p\in \overline\Omega : D^2 u(p)=\landa \, {\rm Id} \text{ for some $\landa\in \R$}\} = \{ p\in \Omega: (\Delta u(p))^2= 4\cH(p)\}.
 \end{equation}
Note that in general $(\Delta u(p))^2\geq  4\cH(p)$ for every $p\in \overline\Omega$.

Choose $p_0\in \cU$. Since problem \eqref{serrineq} is invariant by translations in the $(x,y)$-variables, we may assume without loss of generality that $p_0=(0,0)$. We will also assume: \emph{$u$ is not a polynomial of degree at most two} (note that the statement of Theorem \ref{main} is trivial if $u$ is such a polynomial). Observe that since $u$ is real analytic in $\overline{\Omega}$, it can be extended to a real analytic function on an open set $\Omega_0$ containing $\overline\Omega$; in particular, $u$ can be assumed to be well defined and real analytic around $(0,0)$, even if this point lies in the boundary $\parc \Omega$.

It follows from \eqref{weq} and the fact that $\Phi\not\equiv 0$ on open sets that $u$ satisfies in $\Omega$ the equation
\begin{equation}\label{jacero}
J[\Delta u,\cH(u)]=0,
\end{equation}
where we are denoting for $f,g$ of class $C^1$ $$J[f,g]:=f_x g_y-f_y g_x.$$

Indeed, if $J[\Delta u,\cH(u)]\neq 0$ around some point in $\Omega$, then $(s,t):=(\Delta u, \cH(u))$ are local parameters around that point, and from \eqref{weq} we would obtain that $\Phi$ vanishes in an open set of $\R^2$, a contradiction. 

Therefore, by real analyticity, $u$ satisfies \eqref{jacero} on $\Omega_0\supset\overline\Omega$. Thus, there exists a non-constant real analytic function $\sigma(x,y)$ defined in a neighborhood of $(0,0)$, with $\sigma(0,0)=0$, and non-constant real analytic functions of one variable $\alfa(t),\beta(t)$ such that 
 \begin{equation}\label{fifi00} 
 \Delta u= \alfa\circ \sigma, \hspace{1cm} \cH(u)=\beta \circ \sigma.
 \end{equation} Therefore, there exist non-negative real analytic functions $\phi(t),\varphi(t)$ with $\phi(0)=\varphi(0)=0$ (since $(0,0)\in \cU$) such that 
 \begin{equation}\label{fifi}
 (\Delta u)^2 -4 \cH(u) = \phi \circ \sigma, \hspace{1cm} (\Delta u -\Delta u(p_0))^2 = \varphi \circ \sigma.
 \end{equation}
Since $u$ is not a polynomial of degree $\leq 2$, we can define $w(x,y)$ as the first homogeneous polynomial of degree $\geq 3$ in the Taylor series expansion of $u(x,y)$ around $(0,0)$. Since $(0,0)\in \cU$, we have then around $(0,0)$
 \begin{equation}\label{fifii}
 u(x,y) = c_0+ax+by + \frac{\landa}{2} (x^2+y^2) + w(x,y) + \cdots 
 \end{equation}
From here, a simple power series expansion around the origin shows that 
 \begin{equation}\label{fifi2}
 (\Delta u)^2 -4 \cH(u) =  (\Delta w)^2 -4 \cH(w) +\cdots, \hspace{1cm} (\Delta u -\Delta u(p_0))^2 = (\Delta w )^2 + \cdots 
 \end{equation}
By \eqref{fifi}, this implies that both $(\Delta w)^2 -4 \cH(w)$ and $(\Delta w )^2$ are proportional to $\hat{\sigma}(x,y)^l$ for some $l\geq 1$, where $\hat{\sigma}$ is the first non-zero term in the Taylor series expansion of $\sigma$ at $(0,0)$. In particular, since $(\Delta w)^2 -4 \cH(w)$ is not zero, we have 
 \begin{equation}\label{fifi3}
(\Delta w)^2= \mu^2 ( (\Delta w)^2 -4 \cH(w) ) \end{equation}
for some $\mu\in \R$. Once here, a classical algebraic lemma by Hopf \cite{Ho0} shows that if a homogenous polynomial $w(x,y)$ of degree $n+2\geq 3$ satisfies \eqref{fifi3}, then after a rotation in the $(x,y)$ coordinates, one of the following three situations happens, where we denote $\zeta:=x+iy$:

\begin{enumerate}
\item[{\bf (C.1)}]
$\mu=0$, and $w(x,y)= a \,{\rm Re}(\zeta^{n+2})$ for $a\neq 0$.\vspace{0.1cm}
\item[{\bf (C.2)}]
$\mu=1$, and $w(x,y)=a\,x^{n+2}$ for $a\neq 0$.\vspace{0.1cm}
 \item[{\bf (C.3)}]
$\mu=1+1/k$ for some positive integer $k$, and $w(x,y)=a |\zeta|^{2k+2}$, with $n=2k$, for $a\neq 0$.
\end{enumerate}

Let $\landa\in \R$ be the value for which $D^2u(0,0)=\landa\, {\rm Id}$, and let us write 
\begin{equation}\label{defu1}
u(x,y)=c_0+ax+by+ \frac{\landa}{2}(x^2+y^2)+ u_1(x,y),
\end{equation}
where $a,b,c_0\in \R$ and $u_1(x,y)$ satisfies $u_1(0,0)=Du_1(0,0)=0$. Note that also $D^2u_1(0,0)$ vanishes, by definition of $\landa$. Also, by \eqref{jacero}, we have
\begin{equation}\label{jauno}
J[\Delta u_1,\cH(u_1)]=0.
\end{equation}

The next proposition shows that case {\bf (C.3)} above can only happen, even locally, under very restrictive conditions.

\begin{proposition}\label{prop1}
Assume that case {\bf (C.3)} above happens, i.e. $\mu=1+1/k$ holds in \eqref{fifi3}. Then $u_1(x,y)$ is radially symmetric with respect to $(0,0)$, i.e. $u_1(x,y)$ only depends on $\sqrt{x^2+y^2}.$%=f(\varrho)$ for some real analytic function $f$, where $\varrho:=\sqrt{x^2+y^2}$.
\end{proposition}
\begin{proof}
Note that the leading homogeneous polynomial of $u_1$ is $w(x,y)$, of degree $n+2\geq 3$. As we are in case {\bf (C.3)}, we have 
 \begin{equation}\label{w3}
w(x,y)=a \varrho^{n+2},
 \end{equation} 
where $\varrho:=\sqrt{x^2+y^2}$ and $a\neq 0$.

Assume that $u_1$ is not radial. Then we can write $u_1(x,y)=h(\varrho)+\eta(x,y)+\cdots$, where $\eta(x,y)$ is a non-radial homogeneous polynomial of degree $m+2>n+2$, and $h(\varrho)$ denotes a radial polynomial of degree less than $m+2$. Denote, for $f,g$ of class $C^2$, the operator $$\{f,g\}:=f_{xx}g_{yy} + f_{yy}g_{xx}-2 f_{xy} g_{xy}.$$ 

\noindent Then, if we write $\psi(x,y):=u_1(x,y)-h(\varrho)$, we have
\begin{equation}\label{corche}\def\arraystretch{1.4}
\begin{array}{lll} J[\Delta u_1,\cH(u_1)] & =& J[\Delta h + \Delta \psi,\cH(h)+\cH(\psi)+\{h,\psi\}]
\\ & =& J[\Delta h,\cH(h)] + J[\Delta h,\{h,\psi\}] + J[\Delta\psi ,\cH(h)] \\ & & + J[\Delta h,\cH(\psi)]+ J[\Delta \psi,\cH(\psi)+\{h,\psi\}].\end{array}
\end{equation} 
Note that $J[\Delta h,\cH(h)]=0$, since both $\Delta h$, $\cH(h)$ are radial functions. Also, the least order term in the series expansions of the right-hand side of \eqref{corche} is given by $J[\Delta w,\{w,\eta\}] + J[\Delta\eta ,\cH(w)]$, which has degree $2n+m-2$ (if it is not identically zero). But now, since $u_1$ satisfies \eqref{jauno}, %, and so we have 
 we obtain
 \begin{equation}\label{jodelta}
J[\Delta w,\{w,\eta\}] + J[\Delta\eta ,\cH(w)]=0.
 \end{equation}

%
%
%By using this decomposition of $u_1$, it is easy to show that the first non-zero term in the series expansion of the function 
% \begin{equation}\label{jota}
% J[\Delta u_1,\cH(u_1)]:=(\Delta u_1)_x \cH(u_1)_y-(\Delta u_1)_y \cH(u_1)_x
% \end{equation}
%is given by $ J[\Delta w,\cH(\eta)]+J[\Delta \eta,\cH(w)]  $; note for this that if $v_i=v_i(\varrho)$, $i=1,2$, are two radial functions, then $J[v_1,v_2]=0$, and so all terms in \eqref{jota} vanish up to the first appearance of the non-radial homogeneous polynomial $\eta(x,y)$.
%
%Since $u_1$ satisfies \eqref{weq}, we have $ J[\Delta u_1,\cH(u_1)]=0$, which implies by the previous arguments that
% \begin{equation}\label{jodelta}
%J[\Delta w,\cH(\eta)]+J[\Delta \eta,\cH(w)] =0.
% \end{equation}
 
\noindent We next compute the left-hand side of \eqref{jodelta}. By \eqref{w3}, we have 
\begin{equation}\label{dhw}
\Delta w =a(n+2)^2 \varrho^n, \hspace{0.5cm} \cH(w)=a^2(n+1)(n+2)^2 \varrho^{2n}.
\end{equation} 
If we write $\eta(x,y)$ in polar coordinates as $\eta=c(\theta) \varrho^{m+2}$, then 
 \begin{equation}\label{laeta}
 \Delta \eta= \varrho^m (c''(\theta)+(m+2)^2 c(\theta)).
 \end{equation}
A longer but also straightforward computation, again changing to polar coordinates, shows that
 \begin{equation}\label{corchete}
 \{w,\eta\}=\frac{a}{2}(n+2)\varrho^{n+m}((n+4)(c''(\theta)+(m+2)^2c(\theta))-2n(m+1)(m+2)c(\theta)).
 \end{equation}
 
\noindent Moreover, if we express the Jacobians in \eqref{jodelta} also in polar coordinates, we obtain
\begin{equation}\label{jodelta2}
(\Delta w)_{\rho} (\{w,\eta\})_{\theta}-(\Delta w)_{\theta} (\{w,\eta\})_{\rho} = -(\Delta \eta)_{\rho} (\cH(w))_{\theta}+(\Delta \eta)_{\theta} (\cH(w))_{\rho}.
\end{equation}

\noindent Note that $\Delta w$ and $\cH(w)$ do not depend on $\theta$, by \eqref{dhw}. With this, a computation from \eqref{jodelta2} using \eqref{dhw}, \eqref{laeta} and \eqref{corchete} shows that there exist positive constants $\alfa_1,\alfa_2>0$ such that $$\alfa_1 c'''(\theta)=\alfa_2 c'(\theta).$$ 

\noindent Since $c'(\theta)$ is a periodic function, we necessarily have then $c'(\theta)=0$, i.e., $c(\theta)$ is constant.  This implies that the homogeneous polynomial $\eta=c(\theta)\varrho^{m+2}$ is radial, which contradicts our hypothesis. This finishes the proof of Proposition \ref{prop1}.

\end{proof}

It might be interesting to remark that Proposition \ref{prop1} and the discussion previous to it implies the following consequence, of local nature, which does not use the boundary conditions:

\begin{corollary}
Let $u(x,y)$ be a real analytic function satisfying $J[\Delta u, \cH(u)]=0$. Assume that near the origin, $u$ has the form $$u(x,y)=\frac{\landa}{2}(x^2+y^2) + w(x,y)+ o(\varrho)^k, \hspace{0.5cm} \varrho :=\sqrt{x^2+y^2},$$ where $w(x,y)$ is a homogeneous polynomial of degree $k\geq 3$ that is neither harmonic nor a power of a linear function $(\alfa x+\beta y)^k$, $(\alfa, \beta)\neq (0,0)$. Then $u$ is radial, i.e. $u=u(\varrho)$.
\end{corollary}

\section{Continuity of eigendirections when $\mu=1$}\label{sec:ap}

In what follows we keep the notation of Section \ref{sec:main}. In particular, $p_0=(0,0)\in \cU$ is a point in $\overline\Omega$ where $D^2u=\landa {\rm Id}$ for some $\landa \in \R$. Recall that we can extend $u$ as a real analytic function to an open set $\Omega_0\supset \overline\Omega$; thus, $u$ is real analytic in a neighborhood of $(0,0)$ even if this point lies in $\parc\Omega$.%Thus, equation \eqref{fifii} holds. 

Also, recall that we can define for each $p\in \overline\Omega\setminus \cU$ the eigenlines $L_1(p),L_2(p)$ of $D^2 u(p)$. Thus, $L_1,L_2$ define two real analytic line fields on $\overline\Omega\setminus \cU$; they are given in coordinates with respect to the basis $(dx,dy)$ as the solutions to 
\begin{equation}\label{lain}
-u_{xy} (dx^2-dy^2)+ (u_{xx}-u_{yy}) dx dy =0.
\end{equation}
These eigenlines naturally extend to $\Omega_0\setminus \cU_0$, where  $$\cU_0:=\{p\in \Omega_0 : D^2 u(p) = \landa {\rm Id} \text{ for some $\landa =\landa(p)\in \R$}\}\supset \cU.$$ 

The next result shows that $L_1,L_2$ can be analytically extended around the origin, if $(0,0)\in \cU$ is in case {\bf (C.2)} above:

\begin{proposition}\label{parabolic}
Assume that $p_0=(0,0)\in \cU$, and that case {\bf (C.2)} above happens at $p_0$, i.e. $\mu=1$ in \eqref{fifi3}. Then, there exists $\ep>0$ such that $\Gamma:=\D(\ep)\cap \cU_0$ is a regular, real analytic curve passing through $(0,0)$, and $L_1,L_2$ extend analytically across $\Gamma$, i.e. they define real analytic line fields on $\D(\ep):=\{\xi\in \R^2 : |\xi|<\ep\}$. Moreover, one of $L_1,L_2$ is tangent to $\Gamma$ at $(0,0)$.
\end{proposition}
\begin{proof}
Let $u_1(x,y)$ be the real analytic function defined by \eqref{defu1}. 
%We start by proving the following auxiliary result.
%
%\begin{lemma}\label{lempar}
%Assume that case ${\bf (C.2)}$ happens for $p_0=(0,0)\in \cU$. Then, there exist real analytic local parameters $(s(x,y),t(x,y))$ with $(s(0,0),t(0,0))=(0,0)$, and real analytic functions $g_1(s,t),g_2(s,t)$, such that 
%\begin{equation}\label{eqlema}
%\Delta u_1 = s^n g_1(s,t), \hspace{1cm} (\Delta u_1)^2-4\cH(u_1) =s^{2n} g_2(s,t),
%\end{equation}
%where $g_2(0,0)=g_1(0,0)^2> 0$; here, $n\geq 1$ is the one in ${\bf (C.2)}$.
%
%Moreover, such parameters $(s,t)$ can be chosen so that the $s$-curves are orthogonal to the curve $s=0$.
%\end{lemma}
%\begin{proof}
By \eqref{defu1} and \eqref{fifii}, the first term in the series expansion of $u_1(x,y)$ around the origin is equal to $w(x,y)= a x^{n+2}$. Let us define $\eta(x,y)$ as the least order homogeneous polynomial in the Taylor series of $u_1(x,y)$ that is \emph{not} divisible by $x^{n+2}$. Thus, in case it exists, its degree is $m+2>n+2$. We will consider three cases:

{\bf Case 1:} \emph{$\eta(x,y)$ does not exist.} Therefore, $u_1(x,y)=x^{n+2} \phi(x,y)$ for some real analytic function $\phi$ around the origin, with $\phi(0,0)=a\neq 0$. Thus, using \eqref{defu1}, a simple computation shows that $$(\Delta u)^2-4\cH(u) = (\Delta u_1)^2 -4 \cH(u_1)= x^{2n} G(x,y),$$ for some real analytic function $G(x,y)$ with $G(0,0)=(2+3n+n^2)a^2>0$ (since $\phi(0,0)=a\neq 0$). This implies by \eqref{nees} that there exists $\ep >0$ such that $\D(\ep)\cap \cU$ coincides with the $x=0$ axis. 

Using \eqref{defu1} in a similar way, the equation \eqref{lain} for the eigenlines $L_1,L_2$ is written as:
\begin{equation}\label{lain5}
x^n \left( \Phi_1(x,y) (dx^2-dy^2) + \Phi_2(x,y) dx dy \right)=0,
\end{equation}
where $$\Phi_1(x,y):= -(n+2)x\phi_y - x^2 \phi_{xy}, \hspace{0.5cm} \Phi_2(x,y):= (n+1)(n+2)\phi+x^2(\phi_{xx}-\phi_{yy}).$$
%x^n\left(-((n+2)\phi_y + x \phi_{xy})(dx^2-dy^2)+((n+1)(n+2)\phi+x^2(\phi_{xx}-\phi_{yy})) \, dx dy \right) =0
%\end{equation}
Obviously, \eqref{lain5} defines for each $x\neq 0$ the same directions as 
 \begin{equation}\label{lain6}
\Phi_1(x,y) (dx^2-dy^2) + \Phi_2(x,y) dx dy =0.
 \end{equation}
Moreover, for $x=0$ and $y$ small enough, \eqref{lain6} is just $dx dy=0$, since $\phi(0,0)=a\neq 0$. As a consequence, the eigenlines $L_1,L_2$ extend analytically across the $x=0$ axis in $\D(\ep)$ for $\ep>0$ small enough. More specifically, at each point of the form $(0,y)\in \D(\ep)$, theses eigenlines are precisely $x=0$ and $y=0$. This proves Proposition \ref{parabolic} in Case 1.

%In that situation, the statement of the Lemma is trivial, simply by taking $(s,t)=(x,y)$. 

{\bf Case 2:} \emph{$\eta(x,y)$ has degree $n+3$}, i.e., $m=n+1$. Since $\eta$ is not divisible by $x^{n+2}$, we have that $\eta_{yy}\neq 0$. Using then that $u_1(x,y)=a x^{n+2}+\eta(x,y)+\cdots$ we see that the lowest non-zero Taylor polynomial of $\cH(u_1)$ has degree $2n+1$. Also, note that the lowest non-zero Taylor polynomial of $\Delta u_1$ has degree $n$.

Recall that $J[\Delta u_1,\cH(u_1)]=0$, by \eqref{jauno}. Thus, there exists a non-constant real analytic function $\sigma_1(x,y)$ defined in a neighborhood of $(0,0)$, with $\sigma_1(0,0)=0$, and non-constant real analytic functions of one variable $\alfa_1(t),\beta_1(t)$ with $\alfa_1(0)=\beta_1(0)=0$, such that  \begin{equation}\label{fini}
 \Delta u_1= \alfa_1\circ \sigma_1, \hspace{1cm} \cH( u_1) = \beta_1 \circ \sigma_1.
 \end{equation}
Since, as explained above, $\Delta u_1$ (resp. $\cH(u_1)$) has at the origin a zero of degree $n$ (resp. $2n+1$), and these two integers are coprime, we deduce from \eqref{fini} that $\sigma_1(x,y)$ has a zero of order exactly one at the origin, i.e. $D\sigma_1(0,0)\neq (0,0)$; to see this, note that the vanishing order of $\sigma_1$ at the origin should necessarily be a divisor of both $n$ and $2n+1$, by \eqref{fini}. This also implies by \eqref{fini} that $\alfa_1$ (resp. $\beta_1$) has at the origin a zero of order $n$ (resp. $2n+1$). Thus, by the implicit function theorem and \eqref{fini}, we can choose local coordinates $(s(x,y),t(x,y))$ with $s(x,y):=\sigma_1(x,y)$ and $t(0,0)=0$, such that 
\begin{equation}\label{eqlema}
\Delta u_1 = s^n g_1(s), \hspace{1cm} (\Delta u_1)^2-4\cH(u_1) =s^{2n} g_2(s),
\end{equation}
where $g_2(0)=g_1(0)^2> 0$. %Clearly, we can choose $t(x,y)$ so that $\parc_t(0,0)$ is collinear to $\parc_x$ or $\parc_y$. 
The second equation in \eqref{eqlema}, together with the fact that $(\Delta u)^2-4\cH(u) =(\Delta u_1)^2-4\cH(u_1) $, shows that for $\ep>0$ small enough, $\D(\ep)\cap \cU_0$ agrees with the $s=0$ curve. So, to prove Proposition \ref{parabolic} in this situation we need to show that $L_1,L_2$ extend analytically across $s=0$, and that one of them is tangent to $s=0$ at the origin.

In order to prove this, let us observe that \eqref{eqlema} implies that the eigenvalues $\mu_1,\mu_2$ of $D^2u_1$ can be written in $\D(\ep)$ as $$\mu_i (s)= s^n \varphi_i(s), \hspace{1cm} i=1,2,$$ with $\varphi_1(0)=0$ and $\varphi_2(0)=g_1(0)\neq 0$ (or viceversa). This implies that $D^2 u_1$ can be written as $s^n A(s,t)$, where $A(s,t)$ is a symmetric $2\times 2$ matrix for each $(s,t)$ close enough to the origin, and such that  $A(0,0)$ is not proportional to the identity.
%for $t$ small enough, $A(0,t)$ is diagonal and has exactly one non-zero eigenvalue (since $g_2(0)=g_1(0)^2>0$). 
Consequently, from \eqref{defu1}, $$D^2 u = \landa \, {\rm Id}_2 +s^n A(s,t).$$
Thus, if we denote by $a_{ij}$ to the coefficients of $A$, we see by \eqref{lain} that the eigenlines $L_1,L_2$ are given as the solutions to the equation 
 \begin{equation}\label{lains7}
 s^n\big (-a_{12} (dx^2-dy^2) + (a_{11}-a_{22}) dx dy \big)=0.
 \end{equation}
 Since $A$ is not proportional to the identity at $(0,0)$, the equation 
  \begin{equation}\label{lainss}
 -a_{12} (dx^2-dy^2) + (a_{11}-a_{22}) dx dy =0
 \end{equation}
 defines two real analytic line fields in $\D(\ep)$ for $\ep>0$ small enough, which by \eqref{lains7} agree with $L_1,L_2$ if $s\neq 0$. In other words, $L_1,L_2$ can be analytically extended
 %$a_{12}(0,t)=0$ and $a_{11}(0,t)-a_{22}(0,t) \neq 0$ for $t$ small, 
%we deduce arguing similarly to Case 1 that $L_1,L_2$ are analytic line fields in $\D(\ep)$, i.e. they extend analytically 
across $s=0$ around the origin, as wished. 

Finally, we prove that $L_1$ or $L_2$ is tangent to $s=0$ at the origin. %,along $s=0$, let us first observe that, by the previous discussion, it suffices to prove that this property holds at the origin (since this process could be equally made for any point in the curve $s=0$, as {\bf (C.2)} holds at all those points). %Recall that $\parc_t(0,0)$ was chosen to be collinear with $\parc_x$ or with $\parc_y$. For definiteness, assume it is collinear with $\parc_y$. Thus, $y_t(0,0)=0$, and so $s_x(0,0)=0$, by the inverse function theorem. That $y_t(0,0)=0$ also implies that the curve $s=0$ is tangent to the $x$-axis at the origin.
From $D^2u_1 =s^nA(s,t)$, we have $(u_1)_{xx}=s^n a_{11}$, $(u_1)_{xy}=s^n a_{12}$ and $(u_1)_{yy}=s^n a_{22}$. Therefore,  
$(s^n a_{11})_y=(s^n a_{12})_x$ and $(s^n a_{12})_y=(s^n a_{22})_x$. If we evaluate these equations at the origin, we obtain $$s_y(0,0) \, a_{11}(0,0)= s_x(0,0) \, a_{12}(0,0), \hspace{0.5cm} s_y(0,0) \, a_{12}(0,0)= s_x(0,0) \, a_{22}(0,0).$$ Or equivalently, by the inverse function theorem, 
 \begin{equation}\label{tagg}
-x_t(0,0) \, a_{11}(0,0)= y_t(0,0) \, a_{12}(0,0), \hspace{0.5cm} -x_t(0,0) \, a_{12}(0,0)= y_t(0,0) \, a_{22}(0,0).
 \end{equation} 
Note that $(x_t(0,0),y_t(0,0))$ is tangent to $s=0$ at the origin. From \eqref{tagg}, we obtain at $(0,0)$
$$-a_{12}(x_t^2-y_t^2)+(a_{11}-a_{22}) x_t y_t = y_t(x_t a_{11} + y_t a_{12})-x_t(x_t a_{12}+y_t a_{22})=0.$$
%Since $s_x(0,0)=0$ and $s_y(0,0)\neq 0$, we conclude that $a_{12}(0,0)=0$. 
From \eqref{lainss}, this implies that one of the analytic extensions of $L_1$ and $L_2$ is tangent to $s=0$ at the origin, as wished. % to the origin are parallel to the $x=0$ and $y=0$ axes. Since the $s=0$ curve is tangent to the $x$-axis, we conclude that $L_1$ or $L_2$ is tangent to $s=0$ at the origin, as wished. 
This proves Proposition \ref{parabolic} in Case 2.

%the statement of Lemma \ref{lempar} holds.

{\bf Case 3:} \emph{$\eta(x,y)$ has degree $>n+3$}, i.e. $m>n+1$. We prove next that this case cannot happen, what together with the previous two cases will prove Proposition \ref{parabolic}.

Similarly to our arguments in the proof of Proposition \ref{prop1}, let us start by noting that we can write $u_1(x,y)=h(x,y)+\psi(x,y)$, where $h(x,y)$ is a polynomial of degree at most $m+1$ that is divisible by $x^{n+2}$, and $\psi:=u_1-h$ has $\eta(x,y)$ as the homogeneous polynomial of lowest degree in its series expansion. So, in our conditions, equation \eqref{corche} holds. Next, consider the following facts:

\begin{enumerate}
\item
$J[\Delta u_1,\cH(u_1)]=0$, by \eqref{jauno}. Thus, the left-hand side of \eqref{corche} vanishes.
 \item
$J[\Delta h, \cH(h)]$ is divisible by $x^{3n+1}$, since $h$ is divisible by $x^{n+2}$.
 \item
The lowest term in the right-hand side of \eqref{corche} not coming from $J[\Delta h, \cH(h)]$ is given by 
 \begin{equation}\label{jod2}
 J[\Delta w,\{w,\eta\}] + J[\Delta\eta ,\cH(w)],
 \end{equation} and has degree $2n+m-2$ (if it is not identically zero).
\end{enumerate}
Moreover, since in our situation $w(x,y)=a x^{n+2}$,  we have that $\cH(w)=0$ and that \eqref{jod2} is a constant multiple of $x^{2n-1}\eta_{yyy}$. Since $\eta$ is not divisible by $x^{n+2}$ by hypothesis, we conclude from this discussion that the homogenous polynomial \eqref{jod2} is not divisible by $x^{3n+1}$ (unless it is identically zero). Thus, adding this information to the three facts above, we conclude by \eqref{corche} that $\eta_{yyy}=0$. Consequently,
$$\eta(x,y)=x^m( a_1 x^2 + a_2 xy + a_3 y^2), \hspace{1cm} a_1,a_2,a_3\in \R.$$ Since $m>n+1$ by hypothesis, we see that $x^{n+2}$ divides $\eta(x,y)$, a contradiction. So, Case 3 cannot happen, as claimed.
 \end{proof}

%In order to prove Proposition \ref{parabolic}, we first establish the following auxiliary result:

%
%\begin{remark}\label{rem:borde2}
%Explicar que esto funciona en un punto del borde, también.
%\end{remark}

 \section{Proof of Theorem \ref{main}}\label{imdi}
 Following previous notations, let $p_0=(0,0)$ be a point in $\overline{\Omega}$ where $D^2u=\landa {\rm Id}$ for some $\landa \in \R$, and for which situation ${\bf (C.3)}$ holds. Then, by \eqref{defu1} and Proposition \ref{prop1}, we have that
  \begin{equation}\label{eqp1}
  u(x,y)=ax+by+c_0+v(\sqrt{x^2+y^2}),
  \end{equation} 
globally on $\overline{\Omega}$ (by analyticity), where $v=v(r)$ is a real analytic function, and $a,b,c_0$ are real constants. 

If $a=b=0$, then from \eqref{eqp1} we see that $u=u(\varrho)$, i.e. $u$ is a radial function with respect to the origin. From here, it is easy to check from the overdetermined conditions in \eqref{serrineq} that $\Omega$ is a disk, and the result follows.

Assume next that $(a,b)\neq (0,0)$. Up to a rotation in the $(x,y)$-coordinates, we can assume that $b=0$, i.e. that 
   \begin{equation}\label{eqp12}
  u(x,y)=ax+c_0+v(\sqrt{x^2+y^2}).
  \end{equation} 
Moreover, let us observe that if $u$ is a solution to \eqref{serrineq} on $\Omega$, and $t\neq 0$, then the function $u_t(x,y):= u(tx,ty)/t^2$ is a solution to \eqref{serrineq} on $\Omega_t:= t \Omega$, for the boundary constant $c_t:=c/|t|$. Using this transformation, it is clear that we can assume without loss of generality that $a=1$ holds in \eqref{eqp12}, i.e. that
   \begin{equation}\label{eqp2}
  u(x,y)=x+c_0+v(\sqrt{x^2+y^2}).
  \end{equation} 
We will keep denoting by $\Omega$ the corresponding rotated and dilated domain in the plane; that is, $\Omega$ will be the simply connected planar domain for which \eqref{serrineq} holds for $u$ as in \eqref{eqp2}. By \eqref{eqp2}, we have $$|Du|^2= \frac{1}{\varrho^2}\left((\varrho+xv'(\varrho))^2 + y^2 v'(\varrho)^2\right)=1+v'(\varrho)^2 + \frac{2xv'(\varrho)}{\varrho}.$$ 

\noindent Since $u=0$ along $\parc \Omega$, we conclude from \eqref{eqp2} and the above equation that 
 $$|Du|^2= 1+v'(\varrho)^2-\frac{2(c_0+v(\varrho))}{\varrho}v'(\varrho) \hspace{1cm} \text{along $\parc \Omega$.} $$ So, since $|Du|^2=c^2$ along $\parc \Omega$ by the Neumann condition in \eqref{serrineq}, we deduce that $v$ is a solution to the ODE 
  \begin{equation}\label{ode1}
  1+v'(\varrho)^2-\frac{2(c_0+v(\varrho))}{\varrho}v'(\varrho)=c^2.
  \end{equation}
Differentiating \eqref{ode1} we obtain
 \begin{equation}\label{ode2}
 \frac{2(c_0+v(\varrho)-\varrho v'(\varrho))(-v'(\varrho)+\varrho v''(\varrho))}{\varrho^2}=0.
 \end{equation}
So, there are two options. If $c_0+v(\varrho)-\varrho v'(\varrho)=0$, then $v(\varrho)=-c_0 + t\varrho$, with $t\in \R$. Otherwise, we have $-v'(\varrho)+\varrho v''(\varrho)=0$, from where $v(\varrho)=t_1 \varrho^2+t_2$, with $t_1,t_2\in \R$.

In the first case, we have by \eqref{eqp2} that $u(x,y)=x+t\sqrt{x^2+y^2}$, and so its nodal set is contained in the union of two straight lines. Thus, this case is impossible, since $u=0$ on $\parc \Omega$.

In the second case, we have by \eqref{eqp2} that $u(x,y)=c_1 + x + c_2 (x^2+y^2)$ for some $c_1,c_2 \in \R$. Thus, $u(x,y)$ is radial with respect to the point $q_0:=(-1/(2c_2),0)\in \R^2$, and $\Omega$ is a disk centered at $q_0$.

In conclusion, we have proved: \emph{if there exists some point $p_0\in \overline{\Omega}$ such that $D^2 u=\landa {\rm Id}$ for some $\landa \in \R$, and for which case ${\bf (C.3)}$ holds, then Theorem \ref{main} is true.} 

So, to finish the proof, \emph{we assume next that there is no point $p_0\in \overline{\Omega}$ with $p_0\in \cU$ and for which case ${\bf (C.3)}$ holds}, and reach a contradiction. Let $L_1,L_2$ denote, as usual, the line fields given by the eigenlines of $D^2u$. As explained previously, they are well defined and analytic in $\overline{\Omega}\setminus \cU$. Moreover, since by hypothesis there are no points in $\cU$ for which case ${\bf (C.3)}$ happens, and by Proposition \ref{parabolic} the line fields $L_1,L_2$ are well defined and analytic around any point for which case ${\bf (C.2)}$ holds, we deduce that $L_1,L_2$ are well defined and analytic in $\overline{\Omega}\setminus \cU_1$, where $\cU_1:=\{p\in \cU : \text{ case ${\bf (C.1)}$ holds at $p$}\}.$

We next prove that the points in $\cU_1$ are isolated.
%and that the Poincaré-Hopf index of $L_1, L_2$ around these points is negative.
Indeed, let $p_0\in \cU_1$, and assume for simplicity that $p_0=(0,0)$. Again, we recall that $u$ can be extended analytically to a neighborhood of the origin, even if $(0,0)\in \parc \Omega$. Following the notations of Section \ref{sec:main}, let $w$ denote the first non-zero homogeneous polynomial of degree $\geq 3$ of the series expansion of $u$ at $(0,0)$. Since $(0,0)\in \cU_1$, $w$ is a homogeneous harmonic polynomial. By the first equation in \eqref{fifi2}, we have
\begin{equation}\label{eqp3}
(\Delta u)^2-4\cH(u)= -4(w_{xx} w_{yy}-w_{xy}^2) + \cdots.
\end{equation} 
By harmonicity of $w$, we have $w_{xx} w_{yy}-w_{xy}^2<0$ in $\R^2\setminus \{(0,0)\}$. Thus, from \eqref{eqp3} we conclude that $(\Delta u)^2-4\cH(u)>0$ in a punctured neighborhood of $(0,0)$. In particular, $(0,0)$ is isolated in $\cU_1$, and the line fields $L_1,L_2$ are well defined around $(0,0)$, with an isolated singularity at the origin. 

This fact together with the boundary conditions imply the following

{\bf Claim:} \emph{$L_1,L_2$ are two analytic line fields with isolated singularities in $\overline\Omega$, the singularities being the points in $\cU_1$. Moreover, each $L_i$, $i=1,2$, is either tangent or normal to $\parc \Omega$ at each $p\in \parc \Omega\setminus \cU_1$}.
\vspace{0.1cm}

\emph{Proof of the Claim:} The fist statement follows directly from the previous discussion. For the second one, let $\gamma(s)=(x(s),y(s))$ be a unit speed parametrization of $\parc \Omega$, and let $\nu(s):=(-y'(s),x(s))$ denote the inner unit normal of $\parc \Omega$. By the overdetermined boundary conditions in \eqref{serrineq}, we have $Du(\gamma(s))=\hat{c} \nu(s)$ for some constant $\hat{c}\in \R$. Differentiating this expression, we obtain that $$D^2 u(\gamma'(s),\nu(s))=0.$$ This implies that the tangent and normal lines to $\parc \Omega$ are eigenlines of $D^2 u$ at every $p\in \parc \Omega$. Thus, each $L_i$ is tangent or normal to $\parc \Omega$ at each $p\in \parc \Omega\setminus \cU$.

Consider finally a point $p\in \parc \Omega$ that lies in $\Sigma:=\{p\in \parc\Omega\cap \cU : p\not\in \cU_1\}.$ Then, case {\bf (C.2)} happens at $p$. By Proposition \ref{parabolic}, we have two possible situations: either $\Sigma$ is a finite set, or $\Sigma=\parc \Omega$. In the first case, again by Proposition \ref{parabolic}, we deduce by continuity that $L_1,L_2$ are well defined at any $p\in\Sigma$, and are tangent or normal to $\parc\Omega$ at that point. Thus, the statement of the Claim follows. In the second case, all points of $\Sigma=\parc \Omega$ are in $\cU$, and by Proposition \ref{parabolic} we deduce that $L_1$ or $L_2$ is globally tangent to $\parc\Omega$. This finishes the proof of the Claim.

\vspace{0.1cm}

In these conditions, it follows from the Claim and a standard application of the Poincaré-Hopf theorem that the sum of all rotation indices of each $L_i$ at the isolated singularities of $\cU_1$ is equal to $1$.

Next, we will compute the rotation index of $L_1,L_2$ at $(0,0)\in \cU_1$. In order to do this, let us first look at the eigenlines $L_1^w, L_2^w$ of $D^2 w$. These eigenlines are given as the solutions to the equation 
 \begin{equation}\label{eqp4}-w_{xy} (dx^2-dy^2)+(w_{xx} - w_{yy})dx dy = 2 {\rm Im} (w_{\zeta \zeta} d\zeta^2)=0,
  \end{equation} where $\zeta:=x+iy$ and $\parc_\zeta:=(\parc_x-i\parc_y)/2$. Since $w=a \zeta^{n+2}$ for $n\geq 1$ by ${\bf (C.1)}$, we see that the eigenlines $L_1^w,L_2^w$ given by \eqref{eqp4} have an isolated singularity at $(0,0)$, and their rotation index at the origin is negative, equal to $-n/2$.
  
 Assume next that $p_0= (0,0)$ is an interior point, i.e. $p_0\in \Omega$. From \eqref{fifii} and the previous arguments, we see then that the eigenlines $L_1,L_2$ of $D^2 u$, given by \eqref{lain}, are arbitrarily well aproximated around the origin by the eigenlines $L_1^w, L_2^w$, given by \eqref{eqp4}. In particular, the rotation index of both line fields $L_1,L_2$ around $(0,0)$ is also negative, and equal to $-n/2$. 
 
 In the case that $(0,0)\in \parc \Omega$, a similar argument shows that the (boundary) index of $L_1$ and $L_2$ at $(0,0)$ coincides with the \emph{half-rotation index} of $L_1^w,L_2^w$ at the origin in a half-plane, which is given by the value $-n/4$.
 
This proves that both $L_1,L_2$ only have isolated singularities of negative index in $\overline\Omega$, what contradicts that the sum of all such indices must be equal to $1$, as explained above. This contradiction proves that there exists a point $p_0\in \cU$ for which ${\bf (C.3)}$ holds. So, by previous arguments, $u$ is radial and $\Omega$ is a disk. This concludes the proof of Theorem \ref{main}.
   
  %wind $n/2$ times around the origin in the counterclock-wise direction. Thus, their Poincaré-Hopf index is harmonic, it follows that $L_1^w$, $L_2^w$ also have an isolated singularity at the origin, and that their Poincaré-Hopf index is given by $$j=-\frac{1}{4\pi} \delta({\rm arg}(w_{zz})) = -n/2<0.$$

% i.e. the graph of $u$ is a rotational cone in $\R^3$ with respect to some vertical axis. So, by the overdetermined conditions in \eqref{serrineq}, we conclude that $\Omega$ is a circle and $u$ is radial

\def\refname{References}

\vskip 0.2cm

\noindent José A. Gálvez

\noindent Departamento de Geometría y Topología,\\ Universidad de Granada (Spain).

\noindent  e-mail: {\tt jagalvez@ugr.es}

\vskip 0.2cm

\noindent Pablo Mira

\noindent Departamento de Matemática Aplicada y Estadística,\\ Universidad Politécnica de Cartagena (Spain).

\noindent  e-mail: {\tt pablo.mira@upct.es}

\vskip 0.4cm

\noindent Research partially supported by MINECO/FEDER Grant no. MTM2016-80313-P and Programa de Apoyo a la Investigacion,
Fundación Séneca-Agencia de Ciencia y Tecnologia
Region de Murcia, reference 19461/PI/14.


\begin{thebibliography}{9}

%\bibitem{A1} A.D. Alexandrov, Uniqueness theorems for surfaces
%in the large, I, {\it Vestnik Leningrad Univ.} {\bf 11} (1956),
%5--17. (English translation: {\it Amer. Math. Soc. Transl.}  {\bf 21} (1962), 341--354).

%\bibitem{Be} L. Bers, Remark on an application of pseudoanalytic functions, {\it Amer. J. Math.} {\bf 78} (1956), 486--496.

%\bibitem{BHS} C. Bianchini, A. Henrot, P. Salani, An overdetermined problem with non-constant boundary condition, {\it Interfaces Free Bound.} {\bf 16} (2014), 215--241.

%\bibitem{BGHT} B. Brandolini, N. Gavitone, C. Nitsch, C. Trombetti, Characterization of ellipsoids through an overdetermined boundary value problem of Monge-Ampère type, {\it J. Math. Pures Appl.} {\bf 101} (2014), 828--841.
%
%\bibitem{BNST} B. Brandolini, C. Nitsch, P. Salani, C. Trombetti, Serrin-type overdetermined problems: an alternative proof, {\it Arch. Rational Mech. Anal.} {\bf 190} (2008), 267--280.

%\bibitem{CS} A. Cianchi, P. Salani, Overdetermined anisotropic elliptic problems, {\it Math. Ann.} {\bf 345} (2009), 859--881.

\bibitem{DPW} M. Del Pino, F. Pacard, J. Wei. Serrin's overdetermined problem and constant mean curvature
surfaces, {\it Duke Math. J.} {\bf 164} (2015), 2643--2722.

\bibitem{DEP} M. Domínguez-Vázquez, A. Enciso, D. Peralta-Salas, Solutions to the overdetermined boundary problem for semilinear equations with position-dependent nonlinearities, preprint (2017),  arXiv:1711.08649.

\bibitem{EFM} J.M. Espinar, A. Farina, L. Mazet, $f$-extremal domains in hyperbolic space, preprint (2015), arxiv:1511.02659.

\bibitem{EM}J.M. Espinar, L. Mazet, Characterization of $f$-extremal disks, {\it J. Differential Equations},  {\bf 266} (2019), 2052--2077.


%\bibitem{ES} C. Enache, S. Sakaguchi, Some fully nonlinear elliptic boundary value problems with ellipsoidal free boundaries, {\it Math. Nachr.} {\bf 284} (2011), 1872--1879.

\bibitem{FV} A. Farina, E. Valdinocci, On partially and globally overdermined problems of elliptic type, {\it Amer. J. Math} {\bf 135} (2013), 1699--1726.
%
%\bibitem{FV2}A. Farina, E. Valdinocci, Flattening results for elliptic PDEs in unbounded domains with applications to overdetermined problems, {\it Arch. Rat. Mech. Anal} {\bf 195} (2010), 1025--1058.
%
\bibitem{FMV} A. Farina, L. Mari, E. Valdinocci, Splitting theorems, symmetry results and overdetermined problems for Riemannian manifolds, {\it Comm. Partial Diff. Equations} {\bf 38} (2013), 1818--1862.
%
%\bibitem{GG} B. Guan, P. Guan, Convex hypersurfaces of
%prescribed curvature, {\it Ann. Math.} {\bf 156} (2002), 655--674.
%
%\bibitem{GM} J.A. Gálvez, P. Mira, A Hopf theorem for non-constant mean curvature and a conjecture by A.D. Alexandrov, {\it Math. Ann.} {\bf 366} (2016), 909--928.

\bibitem{GM3} J.A. Gálvez, P. Mira, Uniqueness of immersed spheres in three-manifolds, {\it J. Diff. Geom.}, to appear. arXiv:1603.07153.

%\bibitem{GM4} J.A. Gálvez, P. Mira, Uniqueness of immersed spheres in three-manifolds II. Applications (in preparation, 2015).

%\bibitem{GT} Gilbard, Trudinger

%\bibitem{HW1} P. Hartman, A. Wintner, Umbilical points and $W$-surfaces, {\it Amer. J. Math.} {\bf 76} (1954), 502--508.

\bibitem{HHP} F. Helein, L. Hauswirth, F. Pacard, A note on some overdetermined problems, {\it Pacific  J. Math.} {\bf
250}
(2011), 319--334.

\bibitem{Ho0} H. Hopf, Uber Flachen mit einer Relation zwischen den Hauptkrummungen, {\it Math. Nachr.} {\bf 4} (1951), 232--249.

%\bibitem{Ho} H. Hopf. {\it Differential Geometry in the Large, volume 1000 of
%Lecture Notes in Math.} Springer-Verlag, 1989.
%
%\bibitem{Hs} W.Y. Hsiang, Generalized rotational hypersurfaces of constant mean curvature in the Euclidean spaces. I. {\it J. Diff. Geom.} {\bf 17} (1982), 337--356.
%
%\bibitem{KP} M. Koiso, B. Palmer, Anisotropic umbilic points and Hopf's theorem for surfaces with constant anisotropic mean curvature, {\it Indiana Univ. Math. J.} {\bf 59} (2010), 79--90

%\bibitem{Le} H. Lewy, On differential geometry in the large, I (Minkowski's problem), {\it Trans. Amer. Math.Soc.} {\bf 42} (1938), 258--270.

\bibitem{Mi} P. Mira, Overdetermined elliptic problems in topological disks, {\it J. Differential Equations} {\bf 264} (2018), 6994--7005.

%\bibitem{N} L. Nirenberg, The Weyl and Minkowski problems in differential geometry in the large, {\it Comm. Pure Appl. Math.} {\bf 6} (1953), 337--394.
%
%\bibitem{N2} L. Nirenberg, On nonlinear elliptic partial differential equations and Holder continuity, {\it Comm. Pure Appl. Math.} {\bf 6} (1953), 103--156.
%
%\bibitem{Ni} J.C.C. Nitsche, Stationary partitioning of convex bodies, {\it Arch. Rational Mech. Anal.} {\bf 89} (1985), 1--19.
%
%\bibitem{O} M. Onodera, On the symmetry in a heterogeneous overdetermined problem, {\it Bull. Lond. Math. Soc.} {\bf 47} (2015), 95--100.
%
%\bibitem{Po} A.V. Pogorelov, Regularity of a convex surface with given Gaussian curvature, {\it Mat. Sbornik N.S.} {\bf 31} (1952), 88--103.

\bibitem{new} W.F. Newns, Functional dependence, {\it Amer. Math. Monthly}, {\bf 74} (1967), 911--920.

\bibitem{RRS1} A. Ros, D. Ruiz, P. Sicbaldi, A rigidity result for overdetermined elliptic problems in the plane, {\it Comm. Pure Appl. Math.}, {\bf 70} (2017), 1223--1252. 

\bibitem{RRS2} A. Ros, D. Ruiz, P. Sicbaldi, Solutions to overdetermined elliptic problems in non-trivial exterior domains, {\it J. Eur. Math. Soc.}, to appear, arXiv:1609.03739.

\bibitem{RS} A. Ros, P. Sicbaldi, Geometry and topology of some overdetermined elliptic problems, {\it J. Differential Equations}, {\bf 255} (2013), 951-977.

\bibitem{Se} J. Serrin, A symmetry problem in potential theory, {\it Arch. Rational Mech. Anal.} {\bf 43} (1971), 304--318.

\bibitem{ScS} F. Schlenk, P. Sicbaldi, Bifurcating extremal domains for the first eigenvalue of the Laplacian, {\it Adv. Math.} {\bf 229} (2012), 602--632.

\bibitem{Si} P. Sicbaldi, New extremal domains for the first eigenvalue of the Laplacian in flat tori, {\it Calc. Var. Partial Diff. Equations} {\bf 37} (2010), 329--344.

\bibitem{SS} L. Silvestre, B. Sirakov, Overdetermined problems for fully nonlinear elliptic equations, {\it Calc. Var. Partial Diff. Equations} {\bf 54} (2015), 989--1007.

\bibitem{T} M. Traizet. Classification of the solutions to an overdetermined elliptic problem in the plane,
{\it Geom.
Funct. Anal.}
{\bf 24}
(2014), 690--720.

\bibitem{Vo} K. Voss, Uber geschlossene Weingartensche Flachen, {\it Math. Ann.} {\bf 138} (1959), 42--54.

%\bibitem{WC} G. Wang, C. Xia, A characterization of the Wulff shape by an overdetermined anisotropic PDE, {\it Arch. Rational Mech. Anal.}, {\bf 199} (2011), 99--115.

\bibitem{W} H.F. Weinberger, Remark on the preceeding paper of Serrin, {\it Arch. Rational Mech. Anal.} {\bf 43} (1971), 319--320.

%\bibitem{W} H.C. Wente, Counterexample to a conjecture of H. Hopf, {\it Pacific J. Math.} {\bf 121} (1986), 193--243.

\end{thebibliography}
\end{document}